\documentclass[11pt]{amsart}
\usepackage[utf8]{inputenc}
\usepackage{amsmath}
\usepackage[makeroom]{cancel}
\usepackage{tikz-cd} 
\usepackage{todonotes}
\usepackage{amsfonts, amsthm}
\usepackage{subfig}
\usepackage{mathtools}
\usepackage{amssymb, hyperref, color, graphicx, caption}
\usepackage{geometry}
\usepackage{mdframed}
\usepackage{hyperref}
\usepackage{tikzsymbols}
\usepackage{tcolorbox}
\usepackage{xcolor}   
\usepackage{hyperref}
\hypersetup{
    colorlinks=true, 
    linktoc=all,     
    linkcolor=blue,  
    citecolor=blue
}

\newcommand{\R}{\mathbb{R}}

\newtheorem{thm}{Theorem}

\numberwithin{thm}{section}

\newtheorem{cor}[thm]{Corollary}
\newtheorem{lem}[thm]{Lemma}

\newtheorem{prop}[thm]{Proposition}
\theoremstyle{definition}

\makeatletter
\newtheorem*{rep@theorem}{\rep@title}
\newcommand{\newreptheorem}[2]{%
\newenvironment{rep#1}[1]{%
 \def\rep@title{#2 \ref{##1}}%
 \begin{rep@theorem}}%
 {\end{rep@theorem}}}
\makeatother

\newreptheorem{theorem}{Theorem}

\newreptheorem{lemma}{Lemma}

\begin{document}

\title[S. Das, H. Lang, H. Wan, and N. Xu]{The Distribution of Error Terms of Smoothed Summatory Totient Functions}

\author[S.~Das]{Sanjana Das}
\address{Department of Mathematics \\ 
Massachusetts Institute of Technology \\ 
77 Massachusetts Avenue \\
Cambridge, MA 02139}
\email{\href{mailto:sanjanad@mit.edu}{sanjanad@mit.edu}}

\author[H.~Lang]{Hannah Lang}
\address{Department of Mathematics \\
Harvard University \\
1 Oxford Street \\
Cambridge, MA 02138}
\email{\href{mailto:hlang@college.harvard.edu}{hlang@college.harvard.edu}}

\author[H.~Wan]{Hamilton Wan}
\address{Department of Mathematics \\
Yale University\\
10 Hillhouse Avenue\\
New Haven, CT 06520}
\email{\href{mailto:hamilton.wan@yale.edu}{hamilton.wan@yale.edu}}

\author[N.~Xu]{Nancy Xu}
\address{Department of Mathematics \\
Princeton University \\
304 Washington Road \\
Princeton, NJ 08544}
\email{\href{mailto:nancyx@princeton.edu}{nancyx@princeton.edu}}

\begin{abstract}
    We consider the summatory function of the totient function after applications of a suitable smoothing operator and study the limiting behavior of the associated error term. Under several conditional assumptions, we show that the smoothed error term possesses a limiting logarithmic distribution through a framework consolidated by Akbary--Ng--Shahabi. To obtain this result, we prove a truncated version of Perron's inversion formula for arbitrary Riesz typical means. We conclude with a conditional proof that at least two applications of the smoothing operator are necessary and sufficient to bound the growth of the error term by $\sqrt{x}$. 
\end{abstract}

\maketitle

\section{Introduction}

For each positive integer $n$, the Euler totient function is defined as
\[\phi(n) := n \prod_{\substack{p\mid n \\ p \text{ prime}}} \left(1 - \frac{1}{p}\right).\] Consider the associated summatory function: for any $x \geq 1$, \[F(x) = \sum_{n \leq x} \phi(n).\] Classical methods show that \[F(x) = \frac{3}{\pi^2}x^2 + R(x),\] where $R(x) = o(x^2)$ is the associated error term. The limiting behavior of $R(x)$ as $x \to \infty$ has been studied by many authors. For instance, Walfisz \cite{Wal63} has demonstrated that \[R(x) = O(x(\log x)^{2/3}(\log \log x)^{4/3})\] unconditionally. On the other hand, Montgomery \cite{Mon87} proved that $R(x)$ admits a superlinear lower bound: \[R(x) = \Omega(x\sqrt{\log \log x}).\]
Unlike in the case of many other summatory functions of arithmetic functions, such as the M\"obius function or the $k$-free indicator function, the error term $R(x)$ is not solely dictated by the zeros of the Riemann zeta function; instead, it is split into the sum of an \emph{arithmetic} error term $R^{\mathrm{AR}}(x)$ and an \emph{analytic} error term $R^{\mathrm{AN}}(x)$, where only the latter can be written as a sum over the zeros of the Riemann zeta function. Kaczorowski and Wiertelak \cite{KW10} proved under the assumption of the Riemann hypothesis that the analytic error term satisfies the bound \[R^{\mathrm{AN}}(x) = O_\varepsilon(x^{1/2 + \varepsilon})\] for any $\varepsilon > 0$; on the other hand, the arithmetic error term is much larger, with \[R^{\mathrm{AR}}(x) = \Omega_{\pm}(x\sqrt{\log \log x}).\] In this sense, the arithmetic error term is the main contributor to Montgomery's lower bound on the growth of $R(x)$. 

The existence of the arithmetic error term makes it difficult to study $R(x)$. For example, the general framework of \cite{ANS14} that was used to prove the existence of limiting logarithmic distributions for various summatory functions of arithmetic functions or their associated error terms does not apply to $R(x)$. However, as shown by Kaczorowski and Wiertelak \cite{KWsmooth10}, it is possible to make the error term more well-behaved by applying a smoothing operator. Let $\delta$ be the smoothing operator defined by \[[\delta(f)](x) := \int_0^x \frac{f(t)}{t}\, dt\] for any locally Lebesgue integrable function $f : \mathbb{R}^+ \to \mathbb{R}$ satisfying \[\int_0^1 |f(t)|\,|\log t|^N \, \frac{dt}{t} < \infty \] for every integer $N \geq 1$. Then, for each positive integer $k$, we can consider the $(k - 1)$-fold smoothed summatory function \[F_{k - 1}(x) := \frac{1}{(k - 1)!}\sum_{n \leq x} \phi(n)(\log(x/n))^{k - 1} = \frac{1}{(k - 1)!}[\delta^{k - 1}(F)](x).\] Observe that when $k = 1$, $F_0$ is precisely the summatory function $F$. We will show later that the growth of $F_{k - 1}(x)$ is dominated by the term $\frac{3}{2^{k - 1}\pi^2}x^2$. Accordingly, we define the associated smoothed error term as \[R_{k - 1}(x) := F_{k - 1}(x) - \frac{3}{2^{k - 1}\pi^2}x^2.\] After just one application of the smoothing operator $\delta$, the arithmetic error term vanishes in the sense that the smoothed error term $R_1(x)$ can be written solely as a sum over the zeros of the Riemann zeta function (up to a small adjustment). In fact, Kaczorowski and Wiertelak \cite{KWsmooth10} showed that $R_1(x)$ grows at a much slower rate than $R(x)$: under the Riemann hypothesis, they found that \[R_1(x) = O\left(x^{1/2}\exp\left(\frac{C\log x}{\log \log x}\right)\right)\] for some constant $C$. Therefore, we can expect these smoothed error terms $R_{k-1}(x)$ to be much more amenable to standard approaches for finding limiting logarithmic distributions.

The main purpose of this article is to prove that after a $(k - 1)$-fold application of the smoothing operator for $k \geq 2$, the error term $R_{k - 1}(x)$ possesses a limiting logarithmic distribution under reasonable assumptions. The main assumption throughout the article is the famous Riemann hypothesis, which states that all nontrivial zeros of $\zeta(s)$ have the form $\rho = \frac{1}{2} + i\gamma$, where $\gamma \in \mathbb{R}$. Throughout this article, we will use $\rho$ to refer to a nontrivial zero of $\zeta(s)$, and $\gamma$ to refer to the imaginary part of $\rho$. We also make use of a conjecture formulated independently by Gonek \cite{Gon89} and Hejhal \cite{Hej89} concerning the sum \begin{align*}J_{-\ell}(T) = \sum_{0 < \gamma \leq T} \frac{1}{|\zeta'(\rho)|^{2\ell}}\end{align*} for $\ell \in \mathbb{R}$. This conjecture, which we refer to as the \emph{Gonek--Hejhal conjecture}, states that \begin{align}J_{-\ell}(T) \asymp T(\log T)^{(\ell - 1)^2}\label{eqn:Gonek--Hejhal_conjecture}\end{align} for all $\ell < 3/2$. The conjecture has been resolved in the case $\ell = 0$ by von Mangoldt \cite[pp. 97--100]{Dav80}, who proved that \begin{align}J_0(T) = \sum_{0 < \gamma \leq T} 1 = \frac{T}{2\pi}\log \frac{T}{2\pi e} + O(\log T).\label{vonMangoldt_J0}\end{align} In particular, von Mangoldt's result tells us that there are $O(\log T)$ nontrivial zeros with heights in the interval $[T,T+1]$.

For our purposes, we will need to assume the Gonek--Hejhal conjecture in the case $\ell = 1$, in particular that \[J_{-1}(T) = \sum_{0 < \gamma \leq T} \frac{1}{|\zeta'(\rho)|^2} \ll T.\] Gonek \cite{Gon89} proved $J_{-1}(T) \gg T$, under the Riemann hypothesis and the assumption that all zeros of $\zeta(s)$ are simple, and provided a heuristic argument for (\ref{eqn:Gonek--Hejhal_conjecture}) using Dirichlet polynomial approximations. Note that the Gonek--Hejhal conjecture also implies that all zeros of $\zeta(s)$ are simple, an assumption we use throughout the article. Finally, we will also assume the conjectural bound that \begin{align}\label{eqn:reciprocal_zetaprime_nontrivial_zeros}\frac{1}{|\zeta'(\rho)|} \ll |\rho|^{1/2 - \varepsilon}\end{align} for every nontrivial zero $\rho$. 

As we will see in Section \ref{sec:main_section}, we can express $R_{k-1}(x)$ as a sum involving the quantities $1/\zeta'(\rho)$, and the aforementioned conjectures will play a crucial role in allowing us to control the growth of these summands. These conjectures are widely believed to be true. For instance, Hughes--Keating--O'Connell \cite{HKO00} applied techniques from random matrix theory to find heuristic evidence for a conjecture that agrees with (and further refines) the Gonek--Hejhal conjecture. Gonek \cite{Gon99} used Montgomery's pair correlation conjecture to derive the conjecture $1/|\zeta'(\rho)| \ll_{\varepsilon} |\rho|^{1/3 - \varepsilon}$ for every nontrivial zero $\rho$, a stronger statement than (\ref{eqn:reciprocal_zetaprime_nontrivial_zeros}). However, there is currently not much evidence supporting this bound, so we assume the more conservative conjecture given by (\ref{eqn:reciprocal_zetaprime_nontrivial_zeros}) in the same manner as Akbary, Ng, and Shahabi in \cite{ANS14}.

With these assumptions in mind, we are ready to state the main result of this article. By proving that $e^{-y/2}R_{k-1}(e^y)$ is a $B^2$-almost periodic function, we can deduce that the smoothed error term $R_{k-1}(x)$ possesses a limiting distribution.

\begin{thm}\label{intro_thm:limiting_distribution}
Assume the Riemann hypothesis, that $J_{-1}(T) \ll T$, and that $|\zeta'(\rho)|^{-1} \ll |\rho|^{1/2 - \varepsilon}$. Then the function $$\widetilde{R}_{k-1}(y) := e^{-y/2} R_{k - 1}(e^y)$$ has a limiting distribution when $k \geq 2$. More precisely, there exists a probability measure $\mu_{k-1}$ on the Borel subsets of $\R$ such that for every bounded continuous function $f: \R \to \R$, we have \[\lim_{Y \to \infty} \frac{1}{Y} \int_0^Y f(e^{-y/2}R_{k-1}(e^y))\, dy = \int_\R f(x) \, d\mu_{k-1}(x).\]
\end{thm}

As mentioned earlier, the size of the arithmetic error term prevents us from applying the general framework of \cite{ANS14} to prove the existence of a limiting distribution in the case $k = 1$, i.e., for the unsmoothed summatory totient function. However, it would be interesting to prove the existence of a limiting logarithmic distribution for the normalized \textit{analytic} error term in this case, after finding an explicit expression for it in terms of the zeros of the Riemann zeta function.

Returning to the case $k \geq 2$, we can also explicitly compute the Fourier transform of the limiting distribution $\mu_{k-1}$ under the additional assumption of the \emph{Linear Independence hypothesis}, which states that the imaginary parts of the zeros $\rho$ are linearly independent over $\mathbb{Q}$. Although numerical evidence for this conjecture is sparse (see, for instance, \cite[Table 2, Appendix A]{BT15}), it would be unexpected for the nontrivial zeros to satisfy a linear relation. 

\begin{cor}\label{intro_cor:fourier_transform}
Under the Linear Independence hypothesis and the same assumptions as Theorem \ref{intro_thm:limiting_distribution}, the Fourier transform of $\mu_{k-1}$ exists and is given by \[\widehat{\mu}_{k-1}(y) = \prod_{\gamma > 0}\mathcal{J}\left(\frac{2\zeta\left(-\frac{1}{2} + i\gamma\right)y}{\zeta'\left(\frac{1}{2} + i\gamma\right)\left(\frac{1}{2}+i\gamma\right)^k}\right) ,\] where \[\mathcal{J}(z) = \sum_{r=0}^\infty \frac{(-1)^r(z/2)^{2r}}{(r!)^2}\] is the Bessel function of the first kind. 
\end{cor} 

We conclude with an analysis on the tails of the distribution $\mu_{k-1}$ by studying the maximal possible size of the normalized error term $x^{-1/2}R_{k - 1}(x)$.

\begin{thm}\label{intro_thm:sum_bounded}
Assume the Riemann hypothesis and that $J_{-1}(T) \ll T$. For all positive integers $k \geq 3$, the normalized error term $x^{-1/2}R_{k - 1}(x)$ is bounded. Meanwhile, for $k = 2$, under the addition assumption of the Linear Independence hypothesis, the normalized error term $x^{-1/2}R_1(x)$ is unbounded.
\end{thm} 

In other words, the previous theorem tells us that when $k \geq 3$, we have $R_{k-1}(x) = O(x^{1/2})$. On the other hand, we see that $R_1(x)$ grows faster than $x^{1/2}$ under the Linear Independence hypothesis. The latter result is expected, as \cite[Theorem 1.1]{KWsmooth10} provides the unconditional lower bound $R_1(x) = \Omega_{\pm}(x^{1/2}\log\log\log(x))$. 

The article is structured as follows. In Section \ref{sec:main_section}, we show that $e^{-y/2}R_{k-1}(e^y)$ has a limiting distribution. We begin in Section \ref{subsec:nuts_bolts} by producing a truncated Perron's inversion formula for Riesz typical means, which we will need for estimates in the subsequent sections. In Section \ref{subsec:prelim_estimates}, we provide a preliminary estimate on $R_{k-1}(x)$ by writing it in terms of a sum over the zeros of the Riemann zeta function up to a height $T \geq 2$. We then apply the framework of \cite[Corollary 1.3]{ANS14} in Section \ref{subsec:proof_of_limdist} to show that $e^{-y/2}R_{k-1}(e^y)$ is a $B^2$-almost periodic function and thus possesses a limiting distribution. Finally, in Section \ref{sec:error_bounds}, we show that the normalized error term $x^{-1/2}R_{k-1}(x)$ is bounded for $k \geq 3$, but that under the additional assumption of the Linear Independence hypothesis, this is not the case for $k =2$.

\begin{figure}[h!]
    \centering
    \includegraphics{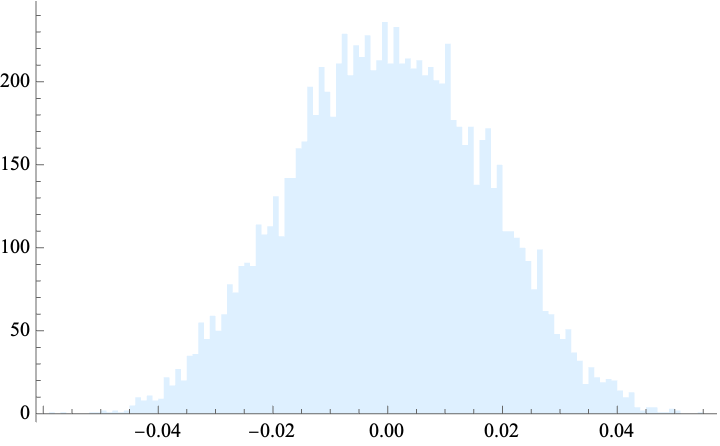}
    \caption{Histogram depicting values of $e^{-y/2}R_{1}(e^y)$ (up to a small error term) for $y=1,\ldots,10000$.}
    \label{fig:k2_10000}
\end{figure}

\begin{figure}[h!]
    \centering
    \includegraphics{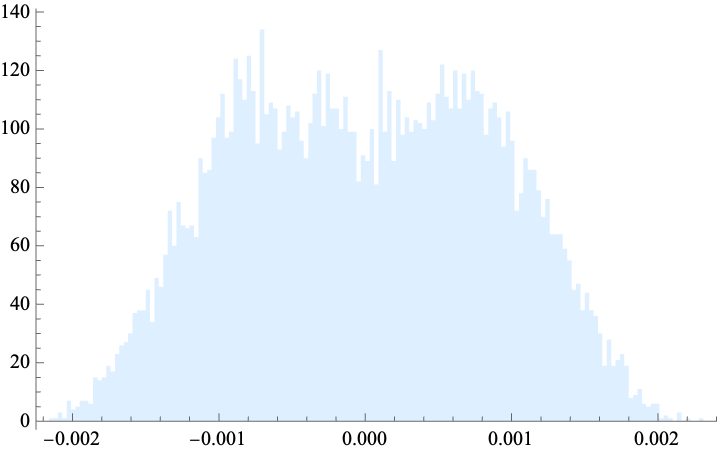}
    \caption{Histogram depicting values of $e^{-y/2}R_{2}(e^y)$ (up to a small error term) for $y=1,\ldots,10000$.}
    \label{fig:k3_10000}
\end{figure}

\subsection*{Acknowledgements} The authors would like to thank Peter Humphries for supervising this project and Ken Ono for his generous support. The authors were participants in the 2022 UVA REU in Number Theory. They are grateful for the support of grants from the National Science Foundation
(DMS-2002265, DMS-2055118, DMS-2147273), the National Security Agency (H98230-22-1-0020), and the Templeton World Charity Foundation.

\section{Limiting Distributions}\label{sec:main_section}

In this section, we show that the function $R_{k - 1}(x)$ has a limiting logarithmic distribution. The general strategy is to write $R_{k - 1}(x)$ in terms of a finite sum over zeros of the Riemann zeta function up to a height $T \geq 2$ and an error term given in terms of $x$ and $T$. We will carefully make our estimates to bound the error term, allowing us to apply \cite[Corollary 1.3]{ANS14} and thereby deduce that $e^{-y/2}R_{k - 1}(e^y)$ is a $B^2$-almost periodic function. 
 
\subsection{Nuts and Bolts} \label{subsec:nuts_bolts}

In this section, we establish several lemmas to facilitate our calculations. We begin with a few estimates involving the Riemann zeta function.

\begin{lem}[{\cite[Lemma 3]{Ng04}}] \label{lem:specific_sequence_heights}
Assume the Riemann hypothesis. Given any $\varepsilon > 0$, for every nonnegative integer $n$, there exists some $T_n \in [n, n + 1]$ such that \[\frac{1}{\zeta(\sigma + iT_n)} = O_\varepsilon(T_n^\varepsilon), \quad \frac{1}{\zeta(\sigma - iT_n)} = O_\varepsilon(T_n^\varepsilon)\] whenever $-1 \leq \sigma \leq 2$. 
\end{lem}

\begin{lem} \label{phragmen lindelof}
Assume the Riemann hypothesis. Given any $\varepsilon > 0$, for all $|T| > 1$, we have $\zeta(\sigma + iT) = O_\varepsilon(|T|^\varepsilon)$ when $1/2 \leq \sigma \leq 3$, and $\zeta(\sigma + iT) = O_\varepsilon(|T|^{1/2 - \sigma + \varepsilon})$ when $-2 \leq \sigma \leq 1/2$. 
\end{lem}

\begin{proof}
For the first part, the Riemann hypothesis implies the Lindel\"of hypothesis, that $\zeta(1/2 + iT) = O_\varepsilon(T^{\varepsilon})$. Meanwhile, we have $\zeta(3 + iT) = O(1)$. The bound for $1/2 \leq \sigma \leq 3$ then follows by the Phragm\'en--Lindel\"of convexity principle. 

For $-2 \leq \sigma \leq 1/2$, the functional equation for the zeta function gives \[\zeta(s) = \frac{\pi^{(s - 1)/2}\Gamma((1 - s)/2)}{\Gamma(s/2)}\zeta(1 - s) \asymp T^{1/2 - \sigma}\zeta(1 - s)\] by a calculation using Stirling's formula, and the bound follows immediately from the bound for $1/2 \leq \sigma \leq 3$. 
\end{proof}

We will now proceed to prove two technical lemmas that will culminate in a truncated version of Perron's inversion formula for {Riesz typical means}. Following \cite[Chapter 5.1.2]{MV06}, we define, for arbitrary integers $k$ and $x$, the \textit{Riesz typical mean} $M_{k}(a_n;x)$ of a sequence $(a_n)_{n=1}^\infty$ to be \begin{align}M_k(a_n;x) :=\frac{1}{k!} \sum_{n \leq x} a_n(\log(x/n))^k.\end{align} In particular, notice that $F_k(x) = M_k(\phi(n);x)$. 

\begin{lem}\label{lem:a^s/s bounds far from n}
For any $a > 0$ with $a \neq 1$, any $c > 0$, and any integer $k \geq 2$, we have \[\frac{1}{2\pi i}\int_{c - iT}^{c + iT} \frac{a^s}{s^k}\, ds = \mathbf{1}_{a > 1}\frac{(\log a)^{k - 1}}{(k - 1)!} + O\left(\frac{a^c}{T^k\log a}\right).\] 
\end{lem}

\begin{proof}
For some real number $m$, consider the integral of $a^s/s^k$ over the positively-oriented rectangle with vertices at $c - iT$, $c + iT$, $m + iT$, and $m - iT$. In the case where $a > 1$, we will take $m \to -\infty$. The only pole of $a^s/s^k$ inside this contour is at $0$, and we have \[\frac{a^s}{s^k} = \frac{1}{s^k}\left(1 + s\log a + \frac{(s\log a)^2}{2!} + \frac{(s\log a)^3}{3!} + \cdots\right),\] so the residue at $0$ is $(\log a)^{k - 1}/(k - 1)!$. Moreover, we have \[\lim_{m \to -\infty} \int_{m + iT}^{m - iT} \frac{a^s}{s^k}\, ds = 0,\] since $|a^s/s^k|$ approaches $0$ and the length of the segment is fixed. Meanwhile, we have \[\int_{c + iT}^{-\infty + iT} \frac{a^s}{s^k}\, ds \ll \int_{-\infty}^c \frac{a^u}{T^k}\, du = \frac{a^c}{T^k\log a},\] and the same is true for the integral from $-\infty - iT$ to $c - iT$. 
    
The case where $a < 1$ can be proved similarly by considering instead the contour where $m \to +\infty$ and observing that there is no pole in this case. 
\end{proof}

\begin{lem}\label{lem:a^s/s bounds close to n}
For any $\frac{1}{2} \leq a \leq 2$, $c > 0$, $T > 0$, and integer $k \geq 2$, \[\frac{1}{2\pi i}\int_{c - iT}^{c + iT} \frac{a^s}{s^k}\, ds = \mathbf{1}_{a > 1}\frac{(\log a)^{k - 1}}{(k - 1)!} + O_\varepsilon\left(\frac{1}{c^\varepsilon T^{k - 1 - \varepsilon}} + \frac{2^c \cdot c}{T^k}\right)\] for any $\varepsilon > 0$.
\end{lem}

\begin{proof}
Consider the integral over the same contour as in Lemma \ref{lem:a^s/s bounds far from n}. In the case $a \geq 1$, it again suffices to bound the integral \[\int_{c + iT}^{-\infty + iT} \frac{a^s}{s^k}\, ds = \left(\int_{c + iT}^{-c + iT} + \int_{-c + iT}^{-\infty + iT}\right) \frac{a^s}{s^k}\, ds.\] For the first segment, we have $|a^s| \leq 2^c$ and $|s^k| \geq T^k$, so \[\int_{c + iT}^{-c + iT} \frac{a^s}{s^k}\, ds \ll \frac{2^c\cdot c}{T^k}.\] For the second segment, we have $|a^s| \leq 1$ and $|s^k| \geq T^{k - 1 - \varepsilon}u^{1 + \varepsilon}$, where $u = -\Re(s)$, which means \[\int_{-c + iT}^{-\infty + iT} \frac{a^s}{s^k}\, ds \ll \frac{1}{T^{k - 1 - \varepsilon}}\int_c^\infty \frac{1}{u^{1 + \varepsilon}}\, du \ll \frac{1}{T^{k - 1 - \varepsilon}c^{\varepsilon}}.\] Similarly, in the case $a \leq 1$ it suffices to bound \[\int_{c + iT}^{\infty + iT} \frac{a^s}{s^k}\, ds \ll \frac{1}{T^{k - 1 - \varepsilon}}\int_c^\infty \frac{1}{u^{1 + \varepsilon}}\, du \ll \frac{1}{T^{k - 1 - \varepsilon}c^{\varepsilon}}. \qedhere\] 
\end{proof}

With the previous two lemmas in place, we are ready to prove a truncated version of Perron's inversion formula for arbitrary {Riesz typical means.} In this article, we will specialize this lemma to the case where the coefficients of the Dirichlet series are given by $\phi(n)$.

\begin{lem}\label{lem:generalized_perron_truncated}
Let $\alpha(s) = \sum_{n = 1}^\infty a_nn^{-s}$ be absolutely convergent for $\Re(s) > \sigma_a$, and $|a_n| \ll \Phi(n)$ for a positive non-decreasing function $\Phi(y)$. Then for any integer $k \geq 2$, for any $c > \max(\sigma_a, 0)$, for all $x \geq 2$ and real $T > 0$ we have 
\begin{align*}M_{k-1}(a_n;x) &= \frac{1}{2\pi i}\int_{c - iT}^{c + iT} \alpha(s)\frac{x^s}{s^k}\, ds \\ &+ O_\varepsilon\left(\frac{\Phi(x + 1)}{c^\varepsilon T^{k - 1 - \varepsilon}} + \frac{2^c \cdot c\Phi(x + 1)}{T^k} + \frac{x^c}{T^k}\sum_{n = 1}^\infty \frac{|a_n|}{n^c} + \frac{2^cx\log x\Phi(2x)}{T^k}\right)\end{align*} for any $\varepsilon > 0$.
\end{lem}

\begin{proof}
By the monotone convergence theorem, we have \begin{align}\frac{1}{2\pi i}\int_{c - iT}^{c + iT} \alpha(s)\frac{x^s}{s^k} \, ds = \sum_{n = 1}^\infty a_n\cdot \frac{1}{2\pi i}\int_{c - iT}^{c + iT} \frac{(x/n)^s}{s^k}\, ds.\label{swapped sum}\end{align} The right-hand side of (\ref{swapped sum}) is equal to \[\frac{1}{(k - 1)!}\sum_{n \leq x} a_n(\log(x/n))^{k - 1} + O\left(\sum_{|n - x|\geq 1} |a_n|\frac{(x/n)^c}{T^k|\log(x/n)|}\right) + O_\varepsilon\left(\frac{\Phi(x + 1)}{c^\varepsilon T^{k - 1 - \varepsilon}} + \frac{2^c \cdot c\Phi(x + 1)}{T^k}\right),\] by applying Lemma \ref{lem:a^s/s bounds far from n} for all $n$ with $|n - x| \geq 1$, and Lemma \ref{lem:a^s/s bounds close to n} for the values of $n$ with $|n - x| < 1$ (of which there are at most two). 

To bound the first error term, we split it into two sums --- one over all $n$ with $n \leq x/2$ or $n \geq 2x$, and one over all $x/2 \leq n \leq x$ with $|n - x| \geq 1$. For the first sum, we have $|\log(x/n)| \gg 1$, which gives \[\sum_{\substack{n \leq x/2 \\ \text{or} ~ n \geq 2x}} |a_n|\frac{(x/n)^c}{T^k|\log(x/n)|} \ll \frac{x^c}{T^k} \sum_{n = 1}^\infty \frac{|a_n|}{n^c}.\] For the second, we have $(x/n)^c \leq 2^c$ and $|a_n| \leq \Phi(2x)$; meanwhile, we also have $|\log(x/n)| \asymp |x - n|/x$ (using the fact that $|\log(1 + t)| \asymp |t|$ for $-1/2 \leq t \leq 1$). This gives \[\sum_{\substack{x/2 \leq n \leq x, \\ |n - x| \geq 1}} |a_n|\frac{(x/n)^c}{T^k|\log(x/n)|} \ll \frac{2^c\Phi(2x)}{T^k}\sum_{\ell \leq x} \frac{1}{\ell} \ll \frac{2^cx\log x\Phi(2x)}{T^k}.\qedhere\] 
\end{proof}

\subsection{Preliminary Estimates}\label{subsec:prelim_estimates}

In this subsection, we will derive an expression for $R_{k - 1}(x)$ involving the zeros of the Riemann zeta function up to a certain height $T$ (given a specific sequence of heights $T_n$ as defined in Lemma \ref{lem:specific_sequence_heights}), and an error term depending on $x$ and $T$ using the lemmas in Section \ref{subsec:nuts_bolts}. In what follows, we will fix $c = 9/4$. Moreover, for the remainder of this article, we assume $k \geq 2$. We begin by applying Lemma \ref{lem:generalized_perron_truncated} to the Dirichlet series corresponding to $\phi(n)$ to write our smoothed error term $R_{k-1}(x)$ in terms of an integral on the segment from $c-iT$ to $c+iT$.

\begin{lem} \label{perron inversion}
For all $x \geq 2$ and all $T > 0$, we have \[R_{k - 1}(x) = \frac{1}{2\pi i}\int_{c - iT}^{c + iT} \frac{\zeta(s - 1)}{\zeta(s)}\frac{x^s}{s^k} \, ds - \frac{3}{2^{k - 1}\pi^2}x^2 + O_\varepsilon\left(\frac{x^c}{T^k} + \frac{x}{T^{k - 1 - \varepsilon}}\right)\] for any $\varepsilon > 0$. 
\end{lem}

\begin{proof}
Note that the Dirichlet series corresponding to $\phi$ is \[\sum_{n = 1}^\infty \frac{\phi(n)}{n^s} = \prod_{p ~ \text{prime}} \left(1 + \sum_{k = 1}^\infty \frac{p^{k - 1}(p - 1)}{p^{ks}}\right) = \frac{\zeta(s - 1)}{\zeta(s)}.\] Now apply Lemma \ref{lem:generalized_perron_truncated}, using $a_n = \phi(n)$, $\Phi(y) = y$, and the fact that $c > 2$ is fixed.
\end{proof}

For a height $T \geq 1$ in the sequence defined by Lemma \ref{lem:specific_sequence_heights} and a nonnegative integer $M$, consider the contour $B_{c,T,M}$ given by the segment $L_1$ from $c - iT$ to $c + iT$, the segment $L_2$ from $c + iT$ to $-(2M + 1) + iT$, the segment $L_3$ from $-(2M + 1) + iT$ to $-(2M + 1) - iT$, and the segment $L_4$ from $-(2M + 1) - iT$ to $c - iT$. Using Lemma \ref{perron inversion}, we can write $R_{k - 1}(x)$ in terms of the integral \[\oint_{B_{c,T,M}} \frac{\zeta(s - 1)}{\zeta(s)}\frac{x^s}{s^k}\, ds,\] and in turn, Cauchy's residue theorem will provide us with an expression for $R_{k - 1}(x)$ in terms of the zeros of $\zeta(s)$ up to height $T$. 

We now provide a brief overview of the remainder of this subsection. First, Lemma \ref{lem:disappearing_integral} will allow us to ``ignore" the integral along $L_3$ by sending $M \to \infty$ through the nonnegative integers. Then, Lemmas \ref{lem:neg_inf_to_neg_1/12}, \ref{lem:2_to_c}, and \ref{lem:neg_1/12_to_2} will allow us to bound the integral along $L_2$ and $L_4$ under the assumption that the height $T$ belongs to the subsequence defined in Lemma \ref{lem:specific_sequence_heights}. Finally, combining these results with our estimate from Lemma \ref{perron inversion} will provide us with a preliminary estimate on $R_{k-1}(x)$ in Proposition \ref{prop:estimate_specific_sequence}. With this overview in mind, we proceed to our computations.

\begin{lem} \label{lem:disappearing_integral}
For any $x \geq 2$ and $T \geq 1$, we have \[\lim_{M \to \infty}\int_{-(2M+1)-iT}^{-(2M+1)+iT}\frac{\zeta(s-1)}{\zeta(s)}\frac{x^s}{s^k} \, ds = 0,\] where $M$ tends to infinity over the nonnegative integers.
\end{lem}

\begin{proof}
Applying the change of variables $u := 1 - s$ and the functional equation for the Riemann zeta function gives us \[\frac{\zeta(s - 1)}{\zeta(s)} = \frac{\zeta(-u)}{\zeta(1 - u)} = \frac{2^{-u}\pi^{-1 - u}\sin(-\pi u/2)\Gamma(u + 1)\zeta(u + 1)}{2^{1 - u}\pi^{-u}\sin(\pi(1 - u)/2)\Gamma(u)\zeta(u)}.\] Since $\Im(u)$ is bounded, $|\sin(-\pi u/2)|$ is bounded above. Meanwhile, we can write \[\sin \frac{\pi(1 -u)}{2} = \frac{1}{2i}\left(e^{i\pi(1 - u)/2} - e^{-i\pi(1 - u)/2}\right),\] from which we deduce that $|\sin(\pi(1 - u)/2)|$ is bounded below.

Meanwhile, $\Gamma(u + 1)/\Gamma(u) = u$. Finally, $|\zeta(u + 1)|$ is bounded above and $|\zeta(u)|$ is bounded below, as $\Re(u)$ is bounded above by $1$. Combining these estimates, we have \[\frac{\zeta(s - 1)}{\zeta(s)} \ll u = 1 - s.\] Since $(1 - s)/s \ll 1$, we have \[\int_{-(2M + 1) - iT}^{-(2M + 1) + iT} \frac{\zeta(s - 1)}{\zeta(s)}\frac{x^s}{s^k}\, ds \ll \int_{-(2M + 1) - iT}^{-(2M + 1) + iT} \left|\frac{x^s}{s^{k - 1}}\right|\, |ds| \ll \frac{Tx^{-(2M + 1)}}{(2M + 1)^{k - 1}}.\] As $M \to \infty$, the right-hand side clearly approaches $0$.
\end{proof}

We will now consider the integral along segments $L_2$ and $L_4$. Note that it suffices to consider the integral along $L_4$, as the same estimate will apply to the integral along $L_2$. We break $L_4$ into three parts. First, we bound the integral from $-(2M + 1) - iT$ to $-1/4 - iT$. 

\begin{lem} \label{lem:neg_inf_to_neg_1/12}
For any $x \geq 2$ and $T \geq 1$, we have \[\int_{-(2M + 1) - iT}^{-1/4 - iT} \frac{\zeta(s - 1)}{\zeta(s)}\frac{x^s}{s^k}\, ds = O\left(\frac{1}{T^{k - 1}x^{1/4}\log x}\right).\] 
\end{lem}

\begin{proof}
By the functional equation for the zeta function, \[\frac{\zeta(s - 1)}{\zeta(s)} = \frac{2(2\pi)^{s - 2}\Gamma(2 - s)\sin(\pi(s - 1)/2)\zeta(2 - s)}{2(2\pi)^{s - 1}\Gamma(1 - s)\sin(\pi s/2)\zeta(1 - s)}.\] We have that $\Gamma(2 - s) = (1 - s)\Gamma(1 - s)$, $\sin(\sigma + it) \asymp e^{t}$ for large $t$, and $|\zeta(2 - s)| \ll 1$ and $|\zeta(1 - s)| \gg 1$. Combining these estimates gives \[\frac{\zeta(s - 1)}{\zeta(s)} \ll 1 - s.\] Therefore, we have \[\int_{-(2M + 1) - iT}^{-1/4 - iT} \frac{\zeta(s - 1)}{\zeta(s)}\frac{x^s}{s^k}\, ds \ll \int_{-\infty - iT}^{-1/4 - iT} \left|\frac{x^s(s - 1)}{s^k}\right| \, |ds|.\] The desired result follows from the bounds  \[\left|\frac{s - 1}{s^k} \right| \leq \left(1 + \frac{1}{T}\right)\cdot \frac{1}{T^{k - 1}} \ll \frac{1}{T^{k - 1}} \quad \text{and}\quad \int_{-\infty - iT}^{-1/4 - iT} |x^s| |ds| \ll \int_{-\infty}^{-1/4} x^u \, du = \frac{x^{-1/4}}{\log x}. \qedhere\] 
\end{proof}

\begin{lem} \label{lem:2_to_c} 
For any $x \geq 2$ and $T \geq 1$, we have \[\int_{2-iT}^{c-iT} \frac{\zeta(s-1)}{\zeta(s)}\frac{x^s}{s^k} \, ds = O_\varepsilon\left(\frac{x^c}{T^{k-\varepsilon}}\right)\] for any $\varepsilon > 0$.
\end{lem}

\begin{proof}
By Lemma \ref{phragmen lindelof}, we have $\zeta(s - 1) = O_\varepsilon(T^\varepsilon)$ for $1 \leq \Re(s) - 1 \leq 2$, while $|\zeta(s)|$ is bounded below. Then \[\int_{2 - iT}^{c - iT} \frac{\zeta(s - 1)}{\zeta(s)}\frac{x^s}{s^k}\, ds \ll \frac{T^\varepsilon x^c}{T^k},\] and the desired result follows. 
\end{proof}

To complete our treatment of the integral along $L_4$, we provide an estimate along the segment from $-1/4 - iT$ to $2 - iT$. In what follows, the assumption that $T$ belongs to the sequence specified in Lemma \ref{lem:specific_sequence_heights} plays a crucial role. 

\begin{lem} \label{lem:neg_1/12_to_2}
Assume the Riemann hypothesis. For any $x \geq 2$ and some value of $T$ in each interval $[n, n + 1]$ for positive integers $n$, we have \[\int_{-1/4 - iT}^{2 - iT} \frac{\zeta(s - 1)}{\zeta(s)} \frac{x^s}{s^k} ds =
O_\varepsilon\left(\frac{x^2}{T^{k - \varepsilon}} + \frac{x^{5/8}}{T^{k - 7/4 - \varepsilon}} + \frac{x^{3/2}}{T^{k - 7/8 - \varepsilon}}\right)\] for any $\varepsilon > 0$. 
\end{lem}

\begin{proof}
Using Lemma \ref{lem:specific_sequence_heights}, we choose $T \in [n, n + 1]$ such that $1/\zeta(s) = O(T^{\varepsilon/2})$ for all $s = \sigma - iT$ with $-1 \leq \sigma \leq 2$. We now split the integral into two smaller integrals, where $\Re(s)$ ranges from $-1/4$ to $3/2$, and from $3/2$ to $2$. 

For $3/2 \leq \Re(s) \leq 2$, we have $|\zeta(s - 1)| = O_\varepsilon(T^{\varepsilon/2})$ by Lemma \ref{phragmen lindelof}. It follows that \[\int_{3/2 - iT}^{2 - iT} \frac{\zeta(s - 1)}{\zeta(s)} \frac{x^s}{s^k}\, ds \ll \frac{x^2}{T^{k - \varepsilon}}.\] 

For $-1/4 \leq \Re(s) \leq 3/2$, we have $|\zeta(s - 1)| = O_\varepsilon(T^{3/2 - \Re(s) + \varepsilon/2})$ by Lemma \ref{phragmen lindelof}. We then have \[\int_{-1/4 - iT}^{3/2 - iT} \frac{\zeta(s - 1)}{\zeta(s)}\frac{x^s}{s^k}\, ds = \left(\int_{-1/4 - iT}^{5/8 - iT} + \int_{5/8 - iT}^{3/2 - iT} \right) \frac{\zeta(s - 1)}{\zeta(s)}\frac{x^s}{s^k} \, ds \ll \frac{x^{5/8}}{T^{k - 7/4 - \varepsilon}} + \frac{x^{3/2}}{T^{k - 7/8 - \varepsilon}}. \qedhere\] 
\end{proof}

Having completely bounded our integral along the contour $B_{c, T, M}$, we are now ready to combine the preceding lemmas and state the main result of this subsection. 

\begin{prop}\label{prop:estimate_specific_sequence}
Assume the Riemann hypothesis and that all zeros of $\zeta(s)$ are simple. For any $x \geq 2$ and for some value of $T$ in each interval $[n,n+1]$ for positive integers $n$, we have \begin{align*}R_{k-1}(x) &=  \sum_{|\gamma| < T} \frac{\zeta(\rho-1)}{\zeta'(\rho)}\frac{x^\rho}{\rho^k} + O_\varepsilon\left(\frac{x}{T^{k - 1 - \varepsilon}} + \frac{x^c}{T^{k - \varepsilon}} + \frac{x^{5/8}}{T^{k - 7/4 - \varepsilon}} + \frac{x^{3/2}}{T^{k - 7/8 - \varepsilon}} + (\log x)^{k - 1}\right)\end{align*} for any $\varepsilon > 0$.
\end{prop}

\begin{proof}
Let $M$ be a nonnegative integer, and consider the positively oriented rectangle $B_{c, T, M}$ with vertices at $c - iT$, $c + iT$, $-(2M + 1) + iT$, and $-(2M + 1) - iT$. Then Lemma \ref{perron inversion} can be rewritten as 
\begin{align*}
    R_{k - 1}(x) &= \frac{1}{2\pi i}\oint_{B_{c,T,M}} \frac{\zeta(s - 1)}{\zeta(s)}\frac{x^s}{s^k}\, ds \\
    &- \frac{1}{2\pi i}\left(\int_{c + iT}^{-(2M + 1) + iT} + \int_{-(2M + 1) + iT}^{-(2M + 1) - iT} + \int_{-(2M + 1) - iT}^{c - iT}\right)\frac{\zeta(s - 1)}{\zeta(s)}\frac{x^s}{s^k}\, ds \\
    &- \frac{3}{2^{k - 1}\pi^2}x^2 + O\left(\frac{x^c}{T^k} + \frac{x}{T^{k - 1 - \varepsilon}}\right).
\end{align*}
The function $\zeta(s - 1)x^s/\zeta(s)s^k$ has poles at $0$, $2$, and the zeros of the zeta function in the interior of $B_{c,T,M}$, so Cauchy's residue theorem gives that \begin{align*}
    \frac{1}{2\pi i}\oint_{B_{c,T,M}} \frac{\zeta(s - 1)}{\zeta(s)}\frac{x^s}{s^k}\, ds &= \sum_{|\gamma| < T} \frac{\zeta(\rho - 1)}{\zeta'(\rho)}\frac{x^\rho}{\rho^k} + \frac{3}{2^{k - 1}\pi^2}x^2 \\
    &+ \sum_{m = 1}^M \frac{\zeta(-2m - 1)}{\zeta'(-2m)}\frac{x^{-2m}}{(-2m)^2} + O((\log x)^{k - 1}),
\end{align*}
where $O((\log x)^{k - 1})$ comes from the residue of the integrand at $s = 0$. 

A calculation using the functional equation for the Riemann zeta function gives that \[\left|\frac{\zeta(-2m - 1)}{\zeta'(2m)}\right| \leq \frac{1}{2\pi}(2m + 1) \frac{\zeta(2m + 2)}{\zeta(2m + 1)} \leq \frac{\zeta(2)(2m + 1)}{2\pi},\] so the sum over the trivial zeros converges as $M \to \infty$. 

Finally, when we take $M \to \infty$ through the nonnegative integers, the integral from $-(2M + 1) + iT$ to $-(2M + 1) - iT$ converges by Lemma \ref{lem:disappearing_integral}. Combining the bounds from Lemmas \ref{lem:neg_inf_to_neg_1/12}, \ref{lem:2_to_c}, and \ref{lem:neg_1/12_to_2} gives the desired result. 
\end{proof}

\subsection{Proof of Theorem \ref{intro_thm:limiting_distribution}}\label{subsec:proof_of_limdist}

In this subsection, we use our preliminary estimates to complete the proof of Theorem \ref{intro_thm:limiting_distribution}. First, we need to extend our estimate in Proposition \ref{prop:estimate_specific_sequence} so that it applies to an arbitrary height $T\geq 2$. As we will see in the following proposition, inspired by \cite[Lemma 5]{Ng04}, we will gain an extra error term upon generalizing to all $T \geq 2$, which we require the conjecture $J_{-1}(T) \ll T$ in order to bound.

\begin{prop} \label{main equation}
Assume the Riemann hypothesis and that $J_{-1}(T) \ll T$. For any $x \geq 2$ and $T \geq 2$, we have\begin{align*}R_{k-1}(x) &=  \sum_{|\gamma| < T} \frac{\zeta(\rho-1)}{\zeta'(\rho)}\frac{x^\rho}{\rho^k} \\
&\quad+ O_\varepsilon\left(\frac{x}{T^{k - 1 - \varepsilon}} + \frac{x^c}{T^{k - \varepsilon}} + \frac{x^{5/8}}{T^{k - 7/4 - \varepsilon}} + \frac{x^{3/2}}{T^{k - 7/8 - \varepsilon}} + (\log x)^{k - 1} + \frac{1}{T^{k - 2}}\left(\frac{x\log T}{T}\right)^{1/2}\right)\end{align*}
for any $\varepsilon > 0$.
\end{prop}

\begin{proof}
Pick an arbitrary $T \geq 2$ such that $n \leq T \leq n+1$ for some integer $n \geq 2$. Suppose without loss of generality that $T \leq T_n$, where $T_n \in [n,n+1]$ satisfies Lemma \ref{lem:specific_sequence_heights}. Then \begin{align*}R_{k-1}(x) &=  \sum_{|\gamma| < T_n} \frac{\zeta(\rho-1)}{\zeta'(\rho)}\frac{x^\rho}{\rho^k} - \sum_{T \leq |\gamma| < T_n} \frac{\zeta(\rho-1)}{\zeta'(\rho)}\frac{x^\rho}{\rho^k} \\ &\quad+ O_\varepsilon\left(\frac{x}{T^{k - 1 - \varepsilon}} + \frac{x^c}{T^{k - \varepsilon}} + \frac{x^{5/8}}{T^{k - 7/4 - \varepsilon}} + \frac{x^{3/2}}{T^{k - 7/8 - \varepsilon}} + (\log x)^{k - 1}\right),\end{align*} where we used $T \leq T_n \leq T+1$ to replace all occurrences of $T_n$ in the error term with $T$. 

By the Cauchy--Schwarz inequality, \begin{align*}\left|\sum_{T \leq \gamma \leq T_n} \frac{\zeta(\rho-1)x^\rho}{\rho^{k}\zeta'(\rho)}\right| &\ll \frac{x^{1/2}}{T^{k-1}}\left(\sum_{T \leq \gamma \leq T_n}\frac{1}{|\rho\zeta'(\rho)|^2}\right)^{1/2}\left(\sum_{T \leq \gamma \leq T_n} |\zeta(\rho-1)|^2\right)^{1/2}. \end{align*}

Under the Gonek--Hejhal conjecture $J_{-1}(T) \ll T$, \cite[Lemma 1(ii)]{Ng04} gives us \[\sum_{T \leq \gamma \leq T_n}\frac{1}{|\rho\zeta'(\rho)|^2} \ll \frac{1}{T}.\] Moreover, by Lemma \ref{phragmen lindelof}, \begin{align}\label{eqn:bound_on_zeta_nontrivialzeros} \zeta(\rho-1) \ll \zeta(2-\rho) \cdot \Im(\rho)^{\Re(2-\rho)-1/2} \ll T.\end{align} By (\ref{vonMangoldt_J0}), we have
\[
\left(\sum_{T \leq \gamma \leq T_n} |\zeta(\rho-1)|^2\right)^{1/2} \ll T\left(\sum_{T \leq \gamma \leq T_n} 1\right)^{1/2} \ll T(\log T)^{1/2},
\] 
and the desired result follows.
\end{proof}

Now we have written the error term $R_{k-1}(x)$ in terms of a finite sum over nontrivial zeros of the Riemann zeta function up to an arbitrary height $T \geq 2$ and an error term bounded in terms of $x$ and $T$. We are finally ready to apply the general framework of \cite{ANS14} to demonstrate that $R_{k-1}(x)$ possesses a limiting logarithmic distribution.

{
\renewcommand{\thethm}{\ref{intro_thm:limiting_distribution}}

\begin{thm}
Assume the Riemann hypothesis, that $J_{-1}(T) \ll T$, and that $|\zeta'(\rho)|^{-1} \ll |\rho|^{1/2 - \varepsilon}.$ The function \[\widetilde{R}_{k-1}(y) := e^{-y/2}R_{k-1}(e^y)\] has a limiting distribution when $k \geq 2$. More precisely, there exists a probability measure $\mu_{k-1}$ on the Borel subsets of $\R$ such that for every bounded continuous function $f: \R \to \R$, we have \[\lim_{Y \to \infty} \frac{1}{Y} \int_0^Y f(e^{-y/2}R_{k-1}(e^y))\, dy = \int_\R f(x) \, d\mu_{k-1}(x).\]
\end{thm}
\addtocounter{thm}{-1}
}

\begin{proof}
By Proposition \ref{main equation}, we can write \[\widetilde{R}_{k-1}(y) = 2\Re\left(\sum_{|\gamma| <T} \frac{\zeta(\rho-1)}{\zeta'(\rho)}\frac{e^{i\gamma y}}{\rho^k}\right) + \mathcal{E}(y,T),\] 
where \begin{align}
     \mathcal{E}(y,T) &= O_\varepsilon\left(\frac{e^{y/2}}{T^{k - 1 - \varepsilon}} + \frac{e^{y(c - 1/2)}}{T^{k - \varepsilon}} + \frac{e^{y/8}}{T^{k - 7/4 - \varepsilon}} + \frac{e^y}{T^{k - 7/8 - \varepsilon}} + y^{k - 1}e^{-y/2} + \frac{1}{T^{k - 2}}\left(\frac{\log T}{T}\right)^{1/2}\right). \label{this is messy}
\end{align}

First, we must show that \[\lim_{Y \to \infty} \frac{1}{Y} \int_0^Y |\mathcal{E}(y,e^Y)|^2 \, dy = 0.\] For the sake of convenience, we index each summand in (\ref{this is messy}) as $\mathcal{E}_i$ for $i = 1,\ldots,6$ so that \[\lim_{Y \to \infty} \frac{1}{Y} \int_0^Y |\mathcal{E}(y,e^Y)|^2 \, dy = \lim_{Y \to \infty} O\left(\frac{1}{Y} \int_0^Y \left|\sum_{i=1}^{11} \mathcal{E}_i(y,e^Y)\right|^2\ dy\right).\]

By the Cauchy--Schwarz inequality, we have \[\left|\sum_{i=1}^{6} \mathcal{E}_i(y,e^Y)\right|^2 \leq \left(\sum_{i = 1}^6 |\mathcal{E}_i(y,e^Y)|\right)^2 \leq 6 \sum_{i = 1}^n |\mathcal{E}_i(y,e^Y)|^2,\] so it suffices to show that \[\lim_{Y \to \infty} \frac{1}{Y} \int_0^Y |\mathcal{E}_i(y, e^Y)|^2 \, dy = 0\] for each individual error term $\mathcal{E}_i$ in the sum. This is clearly true for $y^{k - 1}e^{-y/2}$; meanwhile, all other terms are bounded above by $e^{-d_iY}$ for some constant $d_i > 0$ as long as $\varepsilon < \min(5/2 - c, 1/8)$, which gives the desired result. 

Now recall that \[\sum_{T \leq \gamma \leq T+1} 1 \ll \log T.\] Under the assumption of conjecture (\ref{eqn:reciprocal_zetaprime_nontrivial_zeros}), i.e., $|\zeta'(\rho)|^{-1} \ll \gamma^{1/2-\varepsilon},$ we have \[ \frac{1}{\zeta'(\rho)\rho^k} \ll \frac{1}{\gamma^{k}}\frac{1}{\zeta'(\rho)} \ll \frac{1}{\gamma^{k-1/2+\varepsilon}}\] Combining with the bound (\ref{eqn:bound_on_zeta_nontrivialzeros}) on $\zeta(\rho-1)$, we deduce \[\frac{\zeta(\rho-1)}{\zeta'(\rho)\rho^k}  \ll \gamma^{3/2-k-\varepsilon}.\] In particular, notice that $k - 3/2 + \varepsilon > 1/2$ for $k \geq 2$, so  \cite[Corollary 1.3]{ANS14} tells us that $\widetilde{R}_{k-1}(y)$ is a $B^2$-almost periodic function and therefore possesses a limiting distribution.
\end{proof}

Applying \cite[Theorem 1.9]{ANS14} and assuming the Linear Independence hypothesis, we can also explicitly compute the Fourier transform of the limiting distribution $\mu_{k-1}$.

{
\renewcommand{\thethm}{\ref{intro_cor:fourier_transform}}

\begin{cor}
Under the Linear Independence hypothesis and the same assumptions as Theorem \ref{intro_thm:limiting_distribution}, the Fourier transform of $\mu_{k-1}$ exists and is given by \[\widehat{\mu}_{k-1}(y) = \prod_{\gamma > 0}\mathcal{J}\left(\frac{2\zeta\left(-\frac{1}{2} + i\gamma\right)y}{\zeta'\left(\frac{1}{2} + i\gamma\right)\left(\frac{1}{2}+i\gamma\right)^k}\right) ,\] where \[\mathcal{J}(z) = \sum_{r=0}^\infty \frac{(-1)^r(z/2)^{2r}}{(r!)^2}\] is the Bessel function of the first kind. 
\end{cor}
\addtocounter{thm}{-1}
}

\section{Bounds on the Error Term}\label{sec:error_bounds}

In this section, we provide estimates on the size of the normalized error term $x^{-1/2}R_{k-1}(x)$. In particular, we show that the normalized error term is bounded for $k \geq 3$, as the sum \[\sum_{\rho} \frac{\zeta(\rho-1)}{\zeta'(\rho)\rho^k}x^{i\gamma},\] taken over the nontrivial zeros of the Riemann zeta function, converges. On the other hand, in line with the unconditional lower bound provided in \cite[Theorem 1.1]{KWsmooth10}, we show that $x^{-1/2}R_1(x)$ is unbounded under the additional assumption of the Linear Independence hypothesis. 

{
\renewcommand{\thethm}{\ref{intro_thm:sum_bounded}}

\begin{thm}
Assume the Riemann hypothesis and that $J_{-1}(T) \ll T$. For all positive integers $k \geq 3$, the normalized error term $x^{-1/2}R_{k - 1}(x)$ is bounded. Meanwhile, for $k = 2$, under the additional assumption of the Linear Independence hypothesis, the normalized error term $x^{-1/2}R_1(x)$ is unbounded.
\end{thm}
\addtocounter{thm}{-1}
}

\begin{proof}
Under the Riemann hypothesis, Proposition \ref{main equation} gives us \[x^{-1/2}R_{k-1}(x) = \sum_{\rho} \frac{\zeta(\rho-1)}{\zeta'(\rho)\rho^{k}}x^{i\gamma} + o(x),\] so when $k \geq 3$, we will show that this sum converges, and when $k = 2$, we show that this sum diverges under the Linear Independence hypothesis.

First suppose $k = 2$. As we saw in the proof of Proposition \ref{main equation}, we have $\zeta(\rho-1) \asymp \gamma$ for any nontrivial zero $\rho$. It follows that, under the assumption of Linear Independence, \[\limsup_{x \to \infty} \left|\sum_{|\gamma| < T} \frac{\zeta(\rho-1)}{\zeta'(\rho)\rho^{2}}x^{i\gamma}\right| \gg \sum_{|\gamma| < T} \frac{1}{|\zeta'(\rho)\gamma|},\] and the latter series is known to diverge as $T \to \infty$. 

Now suppose $k \geq 3$. Note that it suffices to show that \[\sum_{\rho} \left|\frac{\zeta(\rho-1)}{\zeta'(\rho)\rho^{k}}\right| < \infty.\] Fix $T > 0$, and consider the sum \begin{align}\label{eqn:cauchyschwarzstuff} \left(\sum_{0 < \gamma < T} \left|\frac{\zeta(\rho - 1)}{\zeta'(\rho)\rho^k}\right|\right)^2 \ll \left(\sum_{0 < \gamma < T} \frac{1}{\gamma^2}\right)\left(\sum_{0 < \gamma < T} \left|\frac{1}{\zeta'(\rho)\gamma^{k - 2}}\right|^2\right),\end{align} which follows from the Cauchy--Schwarz inequality and the fact that $\zeta(\rho - 1) \asymp \gamma$. We show that both sums on the right-hand side of (\ref{eqn:cauchyschwarzstuff}) converge as $T \to \infty$.

For the first sum on the right-hand side of (\ref{eqn:cauchyschwarzstuff}), for each $r \in \mathbb{Z}^{+}$, we have \[\sum_{r < \gamma < r + 1} \frac{1}{\gamma^2} \ll \frac{\log r}{r^2},\] which means \[\sum_{\gamma > 0} \frac{1}{\gamma^2} \ll \sum_{r > 0} \frac{\log r}{r^2} < \infty,\] and the first sum converges as $T \to \infty$. 

On the other hand, for the second sum on the right-hand side of (\ref{eqn:cauchyschwarzstuff}) we have, by partial summation,
\begin{align}
    \sum_{0 < \gamma < T} \frac{1}{|\zeta'(\rho)|^2\gamma^{2(k - 2)}} = \frac{1}{T^{2(k - 2)}}\sum_{0 < \gamma < T} \frac{1}{|\zeta'(\rho)|^2} - \int_{14}^T \sum_{0 < \gamma < t} \frac{1}{|\zeta'(\rho)|^2}\cdot (-2k + 4)t^{-2k + 3}\, dt \label{abel summation}
\end{align} for each $T\geq 14$. The desired bound for the right-hand side of (\ref{abel summation}) then follows directly from the assumption that $J_{-1}(T) \ll T$. \qedhere
\end{proof}

\nocite{*}
\bibliography{smoothedtotienterrordistribution}
\bibliographystyle{plain}

\end{document}